\documentclass{aptpub}
\newcommand{\halmos}{\newline\vspace{3mm}\hfill $\Box$}

\authornames{P. \'{S}WI\c{A}TEK AND Z. PALMOWSKI} 
\shorttitle{Loss rate for a general L\'{e}vy process} 

\begin{document}

\title{Loss rate for a general L\'{e}vy process with downward periodic barrier.} 

\authorone[University of Wroc\l aw]{Z. Palmowski} 
\addressone{Mathematical Institute, University of Wroc\l aw, pl.\ Grunwaldzki 2/4, 50-384 Wroc\l aw, Poland, E-mail: zbigniew.palmowski@math.uni.wroc.pl}
\authortwo[University of Wroc\l aw]{P. \'{S}wi\c{a}tek} 
\addresstwo{Mathematical Institute, University of Wroc\l aw, pl.\ Grunwaldzki 2/4, 50-384 Wroc\l aw, Poland, E-mail: przemyslaw.swiatek@math.uni.wroc.pl}

\begin{abstract}
In this paper we consider a general L\'{e}vy process $X$ reflected at downward periodic barrier $A_t$ and constant upper barrier $K$
giving a process $V^K_t=X_t+L^A_t-L^K_t$.  We find the expression for a loss rate defined by $l^K=\mathbb{E} L^K_1$
and identify its asymptotics as $K\to\infty$ when $X$ has light-tailed jumps and $\mathbb{E}X_1<0$.

\end{abstract}



\section{Introduction}
In this paper we consider a general L\'{e}vy process $X$ reflected at downward periodic barrier $A_t=\varphi(t+U)$
(for a periodic nonnegative function $\varphi(t)$ with period length $s$ and $U$ having uniform distribution on $[0,s]$)
and at the constant upper barrier $K$,
giving a process
\begin{equation}\label{jeden}V^K_t=X_t+L^A_t-L^K_t
\end{equation}
being solution of respective Skorohod problem on each period where reflection is within bounded and convex set
(see \cite{Dupuis1, Tanaka}).
In above we assume that $\varphi(t)\in [0,a]$ for some $a<K$. Process $X$ is defined on the filtered space $(\Omega,\mathcal{ F},\{\mathcal{F}_t\}_{t\geq 0}, \mathbb{P})$ with
the natural filtration that satisfies
the usual assumptions of right continuity and completion.
From now on we will also assume that jump measure $\nu$ of $X$ is nonlattice.

In this paper we find the expression for a loss rate defined by
\begin{equation}\label{definitionlossrate}l^K=\mathbb{E} L^K_1,\end{equation}
where $\mathbb{E}$ denotes the expectation when reflected
process is stationary with stationary measure $\pi_K$ (that is $\mathbb{E}(\cdot)=\int_0^\infty \mathbb{E}[\cdot|V_0^K=x]\;\pi_K(dx)$),
 and prove that $l^K\sim De^{-\gamma K}$ as $K\to\infty$, where $\gamma$ solves $\kappa(\gamma)=0$ for a Laplace exponent
$\kappa(\alpha)=\log \mathbb{E} \exp\{\alpha X(1)\}$  (which is well-defined in some set $\Theta$) when $X$ has light-tailed jumps and $\mathbb{E} X_1<0$.

The motivation comes from various queueing and telecommunication models (see \cite{2,4,9,10,12,23,25}).
Applications, where such a reflected L\'{e}vy process considered in this paper is natural, are models
where additionally to the input and output mechanism
modeled by a L\'{e}vy process there is a constant input given by a downward barrier $A_t$. This additional input is not
available on liquid basis, but can only be used after some maturity date has been
reached. We choose that this time lag is fixed and equals the length $s$ of the period of $\varphi$
(for exponential time delay see \cite{KOM}).
For example, in view of internet networking applications, one considers the combined behavior of two
services (e.g. streaming video and some other data). The first input behaves like a L\'{e}vy process.
The other input grows deterministically and can be served only at some fixed time $s$.
The combined workload now behaves like
a L\'{e}vy process reflected at a lower barrier $A_t$.

Fluid models with finite buffers are useful to model systems where losses
are of crucial importance, as in inventory theory and telecommunications. Indeed,
during the recent years real-time applications such as video streaming, interactive games have become increasingly popular among users.
These applications are generally delay sensitive and require some preferential treatment in order to satisfy a desired level of  Quality of Service.
Traditionally finite capacity buffer mechanisms have been employed in the network routers, in which arriving packets are dropped when workload reaches
its maximum capacity. In this paper we analyze the intensity of packets loss given by so-called loss rate given in (\ref{definitionlossrate}).

There has been a great deal of work on overflow probabilities in various fluid and queueing models
but there have been relatively few studies on the loss rate in finite buffer systems. When jumps of the
L\'{e}vy process are heavy-tailed that one might have a hope to find relationship between these
two notions (see \cite{4,10,13} for more classical models).
Still in our model we focus on light-tailed case and therefore we choose Kella-Whitt martingale approach \cite{12}.
In fact we follow the ideas included in seminal paper of Asmussen and Pihlsg{\aa}rd \cite{AP}, where both barriers are constant
(see also \cite{miyazawa} for matrix analytic method). This case corresponds to the assumption that $A(t)\equiv 0$.
Denoting the loss rate by $l^{K,0}$ analyzed in \cite{AP} we of course immediately get bounds:
\begin{equation}\label{bounds}
l^{K,0}\leq l^K \leq l^{K-a,0},
\end{equation}
from which and \cite[Th. 4.1]{AP} it follows e.g. that $\frac{1}{K}\log l^K=-\gamma$.
In this paper we focus on more precise exact asymptotics.

The paper is organized as follows. In Section \ref{sec:pr} we give preliminary results and in Section \ref{sec:main} the main results with proofs.

\section{Preliminaries}\label{sec:pr}

Assume from now on that
$\varphi\in \mathcal{C}^1({\rm int}\, J_k)$ is invertible on some disjoint intervals $J_k$ satisfying $\cup_{k=1}^n J_k=[0,s]$ with $\varphi^\prime(x)\neq 0$ for $x\in {\rm int}\, J_k$.

\begin{lem}
Process $A_t=\varphi(t+U)$ has invariant measure:
$$\xi(dy)=\sum_{k=1}^n\frac{1}{s}\left|h^{'}_k(y)\right|\mathbf{1}_{\varphi({\rm int} J_k)}(y)\,dy,$$ where $h_k$ is a inverse of $\varphi$ on ${\rm int}\, J_k$.
\end{lem}
\begin{proof} It is sufficient to check that for $t\in [0,s)$:
\begin{equation}
\mathbb{P}(A_t\leq x)=\mathbb{P}(\varphi(t+U)\leq x)=\frac{1}{s}\int_0^s\mathbf{1}_{(\varphi(t+u)\leq x)}du
=\frac{1}{s}\int_0^s\mathbf{1}_{(\varphi(u)\leq x)}du\;,\label{ref}
\end{equation}
where the last equality is a consequence of periodicity of $\varphi$.
The second part of the theorem follows from a straightforward arguments concerning a distribution of a piecewise strictly monotone function of r.v.
\halmos\end{proof}

\begin{ex}
The most interesting case for applications is a saw-like lower boundary modeling constant intensity input (with rate $1$ for simplicity):
\begin{equation}\label{saw}\varphi(t)=
t\;{\rm mod}\; a.
\end{equation}
with $0<a<K$. In this case $n=1$, $J_1=[0,a]$, $s=a$ and $\xi(dy)=\frac{dy}{a}\mathbf{1}_{(y\in[0,a])}$.
\end{ex}

\begin{ex}
More complex situation will appear when $\varphi$ is composed from few lines with different slopes:
\begin{equation}\label{exvarphi2}\varphi(t)=\left\{
\begin{array}{lr}
t& \mbox{for $t\in[0,1)$,}\\
1-2(t-1)&\mbox{for $t\in [1, \frac{3}{2})$,}\\
3\left(t-\frac{3}{2}\right)&\mbox{for $t\in [\frac{3}{2}, \frac{5}{2})$.}
\end{array}\right.
\end{equation}
In this case $n=3$, $J_1=[0,1]$, $J_2=[1, \frac{3}{2})$, $J_3=[\frac{3}{2}, \frac{5}{2})$, $s=\frac{5}{2}$ and $\xi(dy)=
\frac{2}{5}\left(\mathbf{1}_{\varphi(J_1)}(y)+\frac{1}{2}\mathbf{1}_{\varphi(J_2)}(y)+\frac{1}{3}\mathbf{1}_{\varphi(J_3)}(y)\right)\,dy.$
\end{ex}

Following arguments of Asmussen \cite[p. 393-394]{2}
we have the following representation of stationary distribution of $V_t^K$.
 \begin{lem}
The stationary distribution $V_\infty^K$ of the two-sided reflected L\'{e}vy process is given by $$\mathbb{P}(V_\infty^K\geq x)=\int_0^{a}
\underbrace{\sum_{k=1}^n\mathbb{P}(X_{\tau}\geq \widehat{A}_\tau^{-z,k} +x)p_k(z)}_{\overline{\pi}_K^z(x)}\xi(dz),$$ where $\tau =\inf\{t\geq 0:X_t \notin [x-K, \widehat{A}_t^{-z,k}+x)\},$ and $\widehat{A}_t^{-z,k}=-\varphi(h_k(z)-t)$ for $z\in \varphi(J_k)$ and
$p_k(z)=\mathbb{P}(U\in J_k|\varphi(U)=z)=\frac{|h_k^{'}(z)|\mathbf{1}_{\varphi(J_k)}(z)}{\sum_{j=1}^n|h_j^{'}(z)|\mathbf{1}_{\varphi(J_j)}(z)}$.
\end{lem}

\begin{lem}\label{Fact}
If $\mathbb{E}|X_1|<\infty$ then $\mathbb{E}L_t^A<\infty$ and $\mathbb{E}L_t^K<\infty$ for each $t\geq 0$.
\end{lem}
\begin{proof}
Note that by (\ref{jeden}) we have $\mathbb{E}L_t^K<\infty$ if $\mathbb{E}L_t^A<\infty$. Condition $\mathbb{E}L_t^A<\infty$ follows from the Wald identity 
applied to the random walk with increments being corrections of $X$ between consecutive visits of the downward barrier (for details see \cite{AP}).
\halmos
\end{proof}

Further we need a little modification of L\'{e}vy exponent $\kappa(\alpha)$. We will treat large and small jumps separately. Let $L$ be a constant that satisfies $L>\max(K,1)$. Then $\kappa(\alpha)$ can be rewritten as follows:
\begin{equation}\theta_L\alpha+\frac{\sigma^2\alpha^2}{2}+\int_{-\infty}^{\infty}[e^{\alpha x}-1-\alpha x\mathbf{1}_{(|x|\leq L)}]\nu(dx),\ \ \alpha\in\Theta,\end{equation}
where $\theta_L=\theta+\int_1^Lx\nu(dx)+\int_{-L}^{-1}x\nu(dx).$

For any process $Y$ we will denote by $\{Y_t^c\}$ its continuous part and the jumps by $\Delta Y_s=Y_s-Y_{s-}$.

We split $\Delta L_t^K$ into two parts, $\underline{\Delta}L_t^K$ and $\overline{\Delta}L_t^K$, corresponding to $\Delta X_s\in[0,L]$ and $\Delta X_s\in(L,\infty)$, respectively, and $\Delta L_t^A$ into  $\underline{\Delta}L_t^A$ and $\overline{\Delta}L_t^A$, corresponding to $\Delta X_s\in[-L,0]$ and $\Delta X_s\in(-\infty,-L)$, respectively. Let $\underline{l}_j^K=\mathbb{E}\sum_{0\leq s\leq 1}\underline{\Delta}L_s^K$, $\ \overline{l}_j^K=\mathbb{E}\sum_{0\leq s\leq 1}\overline{\Delta}L_s^K$, $\ \underline{l}_j^A=\mathbb{E}\sum_{0\leq s\leq 1}\underline{\Delta}L_s^A$, $\ \overline{l}_j^A=\mathbb{E}\sum_{0\leq s\leq 1}\overline{\Delta}L_s^A$. Then $l_j^K=\underline{l}_j^K+\overline{l}_j^K$ and $l_j^A=\underline{l}_j^A+\overline{l}_j^A$. Finally, let
$l_c^K=l^K-l^K_j$ and $l_c^A=l^A-l^A_j$ with $l^A=\mathbb{E}L^A_1$.
\begin{thm}
For $\alpha\in\Theta$,
\begin{eqnarray}\nonumber M_t =&& \kappa(\alpha)\int_0^te^{\alpha V_s^K}ds+e^{\alpha V_0^K}-e^{\alpha V_t^K}+\alpha\int_0^te^{\alpha A_s}dL_s^{A,c}\\&& +\sum_{0\leq s\leq t}e^{\alpha A_s}(1-e^{-\alpha\Delta L_s^A})-\alpha e^{\alpha K}L_t^{K,c} + e^{\alpha K}\sum_{0\leq s\leq t}(1-e^{\alpha\Delta L_s^K})\label{martyngal}\end{eqnarray}
is a zero-mean martingale.
\end{thm}
\begin{proof}
It is well-known that for an $\mathcal{F}_t$-adapted process
$Y_t=\int_0^tdY_s^c+\sum_{0\leq s\leq t}\Delta Y_s$
of locally bounded variation, process
$$K_t=\kappa(\alpha)\int_0^te^{\alpha Z_s}ds+e^{\alpha x}-e^{\alpha Z_t}+\alpha \int_0^te^{\alpha Z_s}dY_s^c+\sum_{0\leq s\leq t}e^{\alpha Z_s}(1-e^{-\alpha\Delta Y_s})$$
is a local martingale whenever $\alpha\in\Theta$, where $Z_t=x+X_t+Y_t$.
Taking $Y_t=L_t^A-L_t^K$ and using Lemma \ref{Fact} to prove that $Y$ has locally bounded variation, we get that
\begin{eqnarray}\nonumber M_t=&&\kappa(\alpha)\int_0^te^{\alpha V_s^K}ds+e^{\alpha V_0^K}-e^{\alpha V_t^K}\\ \nonumber &&+ \alpha\int_0^te^{\alpha V_s^K}dL_s^{A,c}+\sum_{0\leq s\leq t}e^{\alpha V_s^K}(1-e^{-\alpha\Delta L_s^A})\\
\nonumber &&- \alpha \int_0^te^{\alpha V_s^K}dL_s^{K,c}+\sum_{0\leq s\leq t}e^{\alpha V_s^K}(1-e^{\alpha\Delta L_s^K})\end{eqnarray}
is a local martingale. $M_t$ equals (\ref{martyngal}) since $V_s^K=K$ just after a jump of $L_s^K$, and $V_s^K=A_s$ just after a jump of $L_s^A$.
To prove that $\{M_t\}$ is a true martingale it sufficient to show that $\mathbb{E}\sup_{0\leq s\leq t}M_s<\infty$. This follows from the following conditions:
$V_t^K\leq K$, $\mathbb{E}L_t^{A,c},\ \mathbb{E}L_t^{K,c}<\infty$ and $\mathbb{E}\sum_{0\leq s\leq t}|1-e^{\alpha\Delta L_s^{K}}|<\infty$,
$\mathbb{E}\sum_{0\leq s\leq t}|1-e^{-\alpha\Delta L_s^{A}}|<\infty$ (see also the proof of \cite[Prop. 3.1]{AP}).
\halmos\end{proof}

\begin{cor}\label{cor8}
Let $\alpha\in\Theta$. Then $l^K$ satisfies the following equation:
\begin{eqnarray}
\alpha (1-e^{\alpha K})l^K &=& -\kappa(\alpha)\mathbb{E} e^{\alpha V_0^K}+\alpha\mathbb{E}X_1-\alpha e^{\alpha K}\overline{l}_j^K+\alpha\overline{l}_j^A\nonumber\\
\nonumber &&+\frac{\alpha^2}{2}\mathbb{E}\sum_{0\leq s\leq 1}(\underline{\Delta}L_s^K)^2 +\frac{\alpha^2}{2}\mathbb{E}\sum_{0\leq s\leq 1}(\underline{\Delta}L_s^A)^2\\
\nonumber &&-e^{\alpha K}\mathbb{E}\sum_{0\leq s\leq 1}(1-e^{\alpha\overline{\Delta}L_s^K})-\mathbb{E}\sum_{0\leq s\leq 1}e^{\alpha A_s}(1-e^{-\alpha\overline{\Delta}L_s^A})\\
&&-\alpha^2\mathbb{E}\int_0^1A_sdL_s^{A,c}-\alpha^2\mathbb{E}\sum_{0\leq s\leq 1}A_s\underline{\Delta} L_s^A+o(\alpha^2)
\label{7e}
\end{eqnarray}
\end{cor}

\begin{proof} If we take $t=1$ in $M_t$ and use the stationarity of $V_t^K$ we get
\begin{eqnarray}
0=\kappa(\alpha)\mathbb{E}e^{\alpha V_0^K}&+&\alpha\mathbb{E}\int_0^1e^{\alpha A_s}dL_s^{A,c}+\mathbb{E}\sum_{0\leq s\leq 1}e^{\alpha A_s}(1-e^{-\alpha\Delta L_s^A})\\ \nonumber &-&\alpha e^{\alpha K}l_c^K+e^{\alpha K}\mathbb{E}\sum_{0\leq s\leq 1}(1-e^{\alpha\Delta L_s^K}).\label{odn8}
\end{eqnarray}
Moreover, we have
\begin{equation}\label{od9e}\sum_{0\leq s\leq 1}(1-e^{\alpha\Delta L_s^K})=\sum_{0\leq s\leq 1}(1-e^{\alpha\underline{\Delta} L_s^K})+\sum_{0\leq s\leq 1}(1-e^{\alpha\overline{\Delta} L_s^K}),
\end{equation}
\begin{equation}\label{10e}
\sum_{0\leq s\leq 1}e^{\alpha A_s}(1-e^{-\alpha\Delta L_s^A})=\sum_{0\leq s\leq 1}e^{\alpha A_s}(1-e^{-\alpha\underline{\Delta} L_s^A})+\sum_{0\leq s\leq 1}e^{\alpha A_s}(1-e^{-\alpha\overline{\Delta} L_s^A}).
\end{equation}
Applying expansion:
\begin{equation}
e^{\alpha x}=1+\alpha x+\frac{(\alpha x)^2}{2}+\frac{(\alpha x)^3}{6}e^{\theta\alpha x},\quad \theta\in [-1,1]
\end{equation}
to the first parts of the right-hand sides of (\ref{od9e}) and (\ref{10e}) and expansion $e^\alpha=1+\alpha+o(\alpha)$ to
$\alpha\mathbb{E}\int_0^1e^{\alpha A_s}dL_s^{A,c}$ completes the proof.\halmos\end{proof}

We will need also the following observation.

\begin{lem}\label{calkczesci}
We have:
$$\mathbb{E}\int_0^1A_sdL_s^A=\mathbb{E}(A_0)l^A,$$
where $\mathbb{E}A_0=\int_0^ay\,\xi(dy)$.
\end{lem}
\begin{proof}
Note that $$\mathbb{E}L_s^A=s\mathbb{E}L_1^A=sl^A$$ which is consequence of the fact that $L_s^A$ has independent and stationary increments
under invariant starting position of $A$.
The proof follows now from the following:
\begin{eqnarray*}
\mathbb{E}\int_0^1A_sdL_s^A&=&\int_0^a\mathbb{E}\int_0^1A_s^zdL_s^A\xi(dz)\\
&=&\int_0^a\left[\mathbb{E}A_s^zL_s^A|_0^1-\mathbb{E}\int_0^1L_s^AdA_s^z\right]\xi(dz)\\
&=&\int_0^a\left[A_1^zl^A-\int_0^1sl^AdA_s^z\right]\xi(dz)\\
&=&\int_0^a\left[A_1^zl^A-l^A(sA_s^z|_0^1-\int_0^1A_s^zds)\right]\xi(dz)\\
&=&\int_0^al^A\int_0^1A_s^zds\xi(dz)=l^A\int_0^1\int_0^aA_s^z\xi(dz)ds=l^A\mathbb{E}A_0.
\end{eqnarray*}\halmos\end{proof}

Now using Lemma \ref{calkczesci} we can rewrite (\ref{7e}) as follows.
\begin{lem}\label{12lem} As $\alpha \downarrow 0$ we have
\begin{eqnarray}
\alpha (1-e^{\alpha K}+\alpha\mathbb{E}(A_0))l^K &=& -\kappa(\alpha)\mathbb{E} e^{\alpha V_0^K}+\alpha\mathbb{E}X_1-\alpha e^{\alpha K}\overline{l}_j^K+\alpha\overline{l}_j^A\nonumber\\
\nonumber &&+\frac{\alpha^2}{2}\mathbb{E}\sum_{0\leq s\leq 1}(\underline{\Delta}L_s^K)^2 +\frac{\alpha^2}{2}\mathbb{E}\sum_{0\leq s\leq 1}(\underline{\Delta}L_s^A)^2\\
\nonumber &&-e^{\alpha K}\mathbb{E}\sum_{0\leq s\leq 1}(1-e^{\alpha\overline{\Delta}L_s^K})-\mathbb{E}\sum_{0\leq s\leq 1}e^{\alpha A_s}(1-e^{-\alpha\overline{\Delta}L_s^A})\\ &&+\alpha^2\mathbb{E}(A_0)\mathbb{E}X_1+\alpha^2\mathbb{E}\sum_{0\leq s\leq 1}A_s\overline{\Delta} L_s^A+o(\alpha^2).\label{12e}
\end{eqnarray}
\end{lem}

\section{Main results}\label{sec:main}
The first main result gives representation of $l^K$ in terms of basic characteristics of the process $X$ and lower boundary $A$.
\begin{thm}
Let $\{X_t\}$ be a L\'{e}vy process and $l^K$ the loss rate defined in (\ref{definitionlossrate}). If $\int_1^{\infty}y\nu(dy)=\infty$, then $l^K=\infty$, and otherwise
\begin{eqnarray}
l^K&=&\mathbb{E}X_1\left[\frac{1}{K-\mathbb{E}A_0}\int_0^Kx\pi_K(dx)-\frac{\mathbb{E}A_0}{(K-\mathbb{E}A_0)}\right]+\frac{\sigma^2}{2(K-\mathbb{E}A_0)} \\ \nonumber &&+\frac{1}{2(K-\mathbb{E}A_0)}\int_0^a\int_z^K\int_{-\infty}^{\infty}\varphi_K(x,y,z)\nu(dy)\pi_K^z(dx)\xi(dz),
\end{eqnarray}
where
$$\varphi_K(x,y,z)=\left\{ \begin{array}{ll}
-(x-z)^2-2y(x-z)&\mbox{if $y\leq -x+z$},\\
y^2&\mbox{if $-x+z<y<K-x$},\\
2y(K-x)-(K-x)^2&\mbox{if $y\geq K-x$.}
\end{array} \right.$$
\end{thm}
\begin{proof}
The first claim follows immediately if we note that for $\int_1^{\infty}y\nu(dy)=\infty$ and $L>K$ we have: $$l^K\geq \int_0^K\pi_K(dx)\int_L^{\infty}(y-K+x)\nu(dy)\geq\int_L^{\infty}(y-K)\nu(dy)=\infty.$$
The idea of the proof of the second part of the theorem
is based on two steps: {\bf 1)} expanding all terms
on the right-hand side of (\ref{12e}) and
{\bf 2)} sending $L$ to infinity and then $\alpha\downarrow 0$.

{\bf Step 1.} In particular, the first term on r.h.s. of (\ref{12e}) equals:
\begin{eqnarray*}
\kappa(\alpha)\mathbb{E}e^{\alpha V_0^K}=&&\int_0^Ke^{\alpha x}\int_{-\infty}^{\infty}e^{\alpha y}\mathbf{1}_{(|y|\geq L)}\nu(dy)\pi_K(dx)\\
&&-\int_{-\infty}^{\infty}\mathbf{1}_{(|y|\geq L)}\nu(dy)+\alpha\left(\theta_L-\int_0^Kx\int_{-\infty}^{\infty}\mathbf{1}_{(|y|\geq L)}\nu(dy)\pi_K(dx)\right)\\
&&+\alpha^2\left(\theta_L\int_0^Kx\pi_K(dx)+\frac{\sigma^2}{2}+\int_{-L}^{L}\frac{y^2}{2}\nu(dy)\right.\\
&&\left.-\int_0^K\frac{x^2}{2}\int_{-\infty}^{\infty}\mathbf{1}_{(|y|\geq L)}\nu(dy)\pi_K(dx)\right)+o(\alpha^2).
\end{eqnarray*}
Similarly,
\begin{eqnarray*}
\alpha\mathbb{E}X_1=\alpha\theta_L+\alpha\int_{-\infty}^{\infty}y\mathbf{1}_{(|y|\geq L)}\nu(dy),
\end{eqnarray*}
\begin{eqnarray*}
\frac{\alpha^2}{2}\mathbb{E}\sum_{0\leq s\leq 1}(\underline{\Delta}L_s^K)^2=\frac{\alpha^2}{2}\int_0^K\pi_K(dx)\int_{K-x}^{L}(y-K+x)^2\nu(dy)
\end{eqnarray*}
and
\begin{eqnarray*}
\frac{\alpha^2}{2}\mathbb{E}\sum_{0\leq s\leq 1}(\underline{\Delta}L_s^A)^2=\frac{\alpha^2}{2}\int_0^a\int_z^K\int_{-L}^{-x+z}(x+y-z)^2\nu(dy)\pi_K^z(dx)\xi(dz).
\end{eqnarray*}
For $x>0$ denote $\overline{\nu}(x)=\nu((x,\infty))$ and similarly for $x<0$ let $\underline{\nu}(x)=\nu((-\infty,x))$.
Then:
\begin{eqnarray*}
\alpha\overline{l}_j^A&=&
-\alpha\int_{-\infty}^{-L}y\nu(dy)-\alpha\int_0^Kx\pi_K(dx)\underline{\nu}(-L)+\alpha\int_0^a\int_z^Kz\pi_K^z(dx)\xi(dz)\underline{\nu}(-L)
\end{eqnarray*}
and
\begin{eqnarray*}
\alpha e^{\alpha K}\overline{l}_j^K&=&
\alpha\int_L^{\infty}y\nu(dy)+\alpha^2K\int_L^{\infty}y\nu(dy)
\\ &&+(\alpha+\alpha^2K)\left[\int_0^Kx\pi_K(dx)\overline{\nu}(L)-K\overline{\nu}(L)\right]+o(\alpha^2).
\end{eqnarray*}
The other terms equal:
\begin{eqnarray*}
e^{\alpha K}\mathbb{E}\sum_{0\leq s\leq 1}(1-e^{\alpha\overline{\Delta}L_s^K})
&=&(1+\alpha K+\frac{\alpha^2K^2}{2})\overline{\nu}(L)-\int_0^Ke^{\alpha x}\int_L^{\infty}e^{\alpha y}\nu(dy)\pi_K(dx)+o(\alpha^2)
\end{eqnarray*}
and
\begin{eqnarray*}
\mathbb{E}\sum_{0\leq s\leq 1}e^{\alpha A_s}(1-e^{-\alpha\overline{\Delta}L_s^A})
&=&\mathbb{E}\sum_{0\leq s\leq 1}(1-e^{-\alpha\overline{\Delta}L_s^A})+\alpha^2\mathbb{E}\sum_{0\leq s\leq 1} A_s\overline{\Delta}L_s^A+o(\alpha^2)
\end{eqnarray*}
with
\begin{eqnarray*}
\mathbb{E}\sum_{0\leq s\leq 1}(1-e^{-\alpha\overline{\Delta}L_s^A})&=&
\underline{\nu}(-L)-\int_0^Ke^{\alpha x}\int_{-\infty}^{-L}e^{\alpha y}\nu(dy)\pi_K(dx)\\ &&+\int_0^a\alpha z\int_z^Ke^{\alpha x}\int_{-\infty}^{-L}e^{\alpha y}\nu(dy)\pi_K^z(dx)\xi(dz)\\&&-\frac{1}{2}\int_0^a\alpha^2 z^2\int_z^Ke^{\alpha x}\int_{-\infty}^{-L}e^{\alpha y}\nu(dy)\pi_K^z(dx)\xi(dz)+o(\alpha^2).
\end{eqnarray*}


{\bf Step 2.}
If we now rearrange all terms of (\ref{12e}) using above identities and let $L\to\infty$ (note that $\theta_L\to\mathbb{E}X_1$ as $L\to\infty$) we get:
\begin{eqnarray*}
\alpha(1-e^{\alpha K}+\alpha\mathbb{E}A_0)l^K=&&-\mathbb{E}X_1\alpha^2\int_0^Kx\pi_K(dx)-\frac{\sigma^2\alpha^2}{2}\\ &&-\frac{\alpha^2}{2}\int_0^a\int_z^K\int_{-x+z}^{K-x}y^2\nu(dy)\pi_K^z(dx)\xi(dz)\\ &&+\frac{\alpha^2}{2}\int_0^a\int_z^K\int_{K-x}^{\infty}((x-K)^2+2y(x-K))\nu(dy)\pi_K^z(dx)\xi(dz)\\ &&+\frac{\alpha^2}{2}\int_0^a\int_z^K\int_{-\infty}^{-x+z}((x-z)^2+2y(x-z))\nu(dy)\pi_K^z(dx)\xi(dz)\\ &&+\alpha^2\mathbb{E}A_0\mathbb{E}X_1+o(\alpha^2).
\end{eqnarray*}
The proof is completed by dividing both sides of above equation by $\alpha(1-e^{\alpha K}+\alpha\mathbb{E}A_0)$ and sending $\alpha$ to 0.
\halmos\end{proof}

Assume now that there exists $\gamma>0$ ($\gamma\in \Theta$) such that $\kappa(\gamma)=0$.
Define now the new probability measure:
$$\left. \frac{d\mathbb{P}^\gamma}{d\mathbb{P}}\right| _{\mathcal{F}_{t}}=e^{\gamma X_t}$$
for which we have $\mathbb{E}^\gamma X_1=\kappa^{'}(\gamma)>0$ since on $\mathbb{P}^\gamma$ process $X$ is a L\'{e}vy process
with Laplace exponent $\kappa_\gamma(\alpha)=\kappa(\alpha+\gamma)$.
We also need two passage times:
$$\tau^A_z(x)=\inf\{t\geq 0: X_t\geq \widehat{A}_t^{-z}+x\}, \qquad \tau_{-z}^-=\inf\{t\geq 0: X_t<-z\}.$$
Further, let $\tau^A(x)=\inf\{t\geq 0: X_t\geq \widehat{A}_t^{\xi}+x\}$, where
$\widehat{A}_t^{\xi}=\int_0^\infty A_t^{-y}\,\xi(dy)$
and $B^A(x)=X_{\tau^A(x)}-x$.
The second main result concerns the asymptotics of $l^K$ as $K\to\infty$.
\begin{thm}\label{mainas}
Assume that there exists $\gamma>0$ ($\gamma\in \Theta$) such that $\kappa(\gamma)=0$ and $\kappa^{'}(\gamma)<\infty$.
Then there exists random variable
$B^A(\infty)$ such that:
\begin{equation}\label{mainproblem}
\lim_{x\to\infty}\mathbb{E}^\gamma e^{-\gamma B^A(x)}=\mathbb{E}^\gamma e^{-\gamma B^A(\infty)}.\end{equation}
Furthermore, there exists finite constant $D$ such that:
\begin{equation}\label{Cramer}l^K\sim De^{-\gamma K},\qquad \mbox{as $K\to\infty$,}
 \end{equation}
 where we write $f(K)\sim g(K)$ when $\lim_{K\to\infty}f(K)/g(K)=1$, and
 \begin{eqnarray}
 D&=&-\mathbb{E}X_1C_\gamma+\mathbb{E}^\gamma e^{-\gamma B^A(\infty)}\int_0^\infty e^{\gamma x}\mathbb{P}^\gamma(\tau_{-x}^-=\infty)\int_x^\infty \left(1-e^{\gamma(y-x)}\right)\,\nu(dy)\,dx \nonumber\\
 &&+\int_{-\infty}^0\left(y+\gamma^{-1}\left(1-e^{\gamma y}\right)\right)\,\nu(dy)\nonumber\\&&
 +\int_0^\infty \int_0^{a\wedge x}\mathbb{P}(\tau^A_z(x)<\infty)\int_{-\infty}^{-x+z}\left(1-e^{\gamma(x+y-z)}\right)\,\nu(dy)\,\xi(dz)\,dx
\label{D}
 \end{eqnarray}
 with $$C_\gamma=\mathbb{E}e^{\gamma A_s}.$$
\end{thm}
\begin{proof}
The proof is based on the following observation:
\begin{equation}\label{newreplK}
l^K=\frac{e^{\gamma K}}{e^{\gamma K}-C_\gamma}I_1+\frac{1}{e^{\gamma K}-C_\gamma}I_2+\frac{e^{\gamma K}\gamma^{-1}}{e^{\gamma K}-C_\gamma}I_3+\frac{\gamma^{-1}}{e^{\gamma K}-C_\gamma}I_4-\frac{C_\gamma\mathbb{E}X_1}{e^{\alpha K}-C_\gamma},\end{equation}
where
$$I_1=\int_0^K\int_{K-x}^\infty (y-K+x)\nu(dy)\pi_K(dx),$$
$$I_2=\int_0^a\int_0^K\int_{-\infty}^{-x+z} (x+y-z)\nu(dy)\pi_K^z(dx)\xi(dz),$$
$$I_3=\int_0^K\int_{K-x}^\infty (1-e^{\gamma(y-K+x)})\nu(dy)\pi_K(dx),$$
$$I_4=\int_0^a\int_0^K\int_{-\infty}^{-x+z} (1-e^{\gamma(x+y-z)})\nu(dy)\pi_K^z(dx)\xi(dz).$$
Indeed,
note that from equation (7) taking $\alpha=\gamma$ we get:
$$0=\gamma C_\gamma l_c^A-\gamma e^{\gamma K}l_c^K+C_\gamma\mathbb{E}\sum_{0\leq s\leq 1}(1-e^{-\gamma\Delta L_s^A})+e^{\gamma K}\mathbb{E}\sum_{0\leq s\leq 1}(1-e^{\gamma\Delta L_s^K}),$$
where we used fact that $\mathbb{E}\int_0^1e^{\gamma A_s}dL_s^A=C_\gamma l_c^A.$
Let $\epsilon >0$. We split $\Delta L_t^K$ into two parts, $\underline{\Delta}^\epsilon L_t^K$ and $\overline{\Delta}^\epsilon L_t^K$, corresponding to $\Delta X_s\in[0,\epsilon]$ and $\Delta X_s\in(\epsilon,\infty)$, respectively, and $\Delta L_t^A$ into  $\underline{\Delta}^\epsilon L_t^A$ and $\overline{\Delta}^\epsilon L_t^A$, corresponding to $\Delta X_s\in[-\epsilon,0]$ and $\Delta X_s\in(-\infty,-\epsilon)$, respectively.
Now we have:
$$e^{\gamma K}\mathbb{E}\sum_{0\leq s\leq 1}(1-e^{\gamma\underline{\Delta}^\epsilon L_s^K})=e^{\gamma K}\left(-\gamma (l_j^K-\overline{l}_j^K)-\frac{\gamma^2}{2}\mathbb{E}\sum_{0\leq s\leq 1}(\underline{\Delta}^\epsilon L_s^K)^2\right)+o(\epsilon^2),$$
$$\mathbb{E}\sum_{0\leq s\leq 1}(1-e^{-\gamma\underline{\Delta}^\epsilon L_s^A})=\gamma (l_j^A-\overline{l}_j^A)-\frac{\gamma^2}{2}\mathbb{E}\sum_{0\leq s\leq 1}(\underline{\Delta}^\epsilon L_s^A)^2+o(\epsilon^2).$$
Thus,
\begin{eqnarray*}
0&=&\gamma C_\gamma l_c^A-\gamma e^{\gamma K}l_c^K+C_\gamma\mathbb{E}\sum_{0\leq s\leq 1}(1-e^{-\gamma\overline{\Delta}^\epsilon L_s^A})+e^{\gamma K}\mathbb{E}\sum_{0\leq s\leq 1}(1-e^{\gamma\overline{\Delta}^\epsilon L_s^K})\\
&&+\gamma C_\gamma(l_j^A-\overline{l}_j^A)-\frac{\gamma^2}{2}C_\gamma\mathbb{E}\sum_{0\leq s\leq 1}(\underline{\Delta}^\epsilon L_s^A)^2-e^{\gamma K}\gamma(l_j^K-\overline{l}_j^K)-\frac{\gamma^2}{2}\mathbb{E}\sum_{0\leq s\leq 1}(\underline{\Delta}^\epsilon L_s^K)^2\\ &&+o(\epsilon^2).
\end{eqnarray*}
Using the fact that $l^A=l^K-\mathbb{E}X_1$ we have:
\begin{eqnarray*}
l^K(C_\gamma-e^{\gamma K})\gamma&=&\gamma C_\gamma\mathbb{E}X_1+\gamma C_\gamma \overline{l}_j^A-\gamma e^{\gamma K}\overline{l}_j^K-C_\gamma\mathbb{E}\sum_{0\leq s\leq 1}(1-e^{-\gamma\overline{\Delta}^\epsilon L_s^A})\\
&&-e^{\gamma K}\mathbb{E}\sum_{0\leq s\leq 1}(1-e^{\gamma\overline{\Delta}^\epsilon L_s^K})+\frac{\gamma^2}{2}C_\gamma\mathbb{E}\sum_{0\leq s\leq 1}(\underline{\Delta}^\epsilon L_s^A)^2+\frac{\gamma^2}{2}\mathbb{E}\sum_{0\leq s\leq 1}(\underline{\Delta}^\epsilon L_s^K)^2\\ &&+o(\epsilon^2).
\end{eqnarray*}
If we send $\epsilon\to 0$ we get:
\begin{eqnarray*}
l^K(C_\gamma-e^{\gamma K})\gamma=\gamma C_\gamma\mathbb{E}X_1-\gamma C_\gamma I_2-\gamma e^{\gamma K}I_1-C_\gamma I_4-e^{\gamma K}I_3.
\end{eqnarray*}
Now (\ref{newreplK}) follows by dividing by $(C_\gamma-e^{\gamma K})\gamma$.

Note that $I_1$ and $I_3$ are the same like those in \cite[Theorem 3.2]{AP} and $I_2$ and $I_4$ have just additional integral over $\xi(dz)$ and
we should take $x-z$ instead of $x$ under the integrals sign. Using thus the same arguments like in the proof of \cite[Theorem 4.1]{AP} will now completes the proof
when we prove weak convergence (\ref{mainproblem}). To this one can use classical renewal arguments applied to the process $\{X_{\tau^A(an)},n\in N\}$
(by considering {\it ladder height lines} being $\widehat{A}_t$ starting from invariant measure shifted by $a$ from the previous position of ladder process).
\halmos\end{proof}

\begin{ex}
For a stable $M/M/1$ queue, that is for $X_t=\sum_{i=1}^{N_t}\sigma_i-t$ with $\{\sigma_i\}_{\{i\geq 1\}}$ being i.i.d. exponentially r.v.'s with intensity $\mu$ and $N_t$ being a Poisson process
with intensity $\lambda<\mu$, we have $\nu(dx)=\mu\lambda e^{-\mu x}\,dx$ and $\mathbb{P}^\gamma(\tau_{-x}^-=\infty)=1-e^{-\gamma x}$, where $\gamma=\mu-\lambda$,
since considering $X$ on $\mathbb{P}^\gamma$ is equivalent of exchanging intensities of arrival and service processes.
Moreover, choosing saw-like lower boundary given in (\ref{saw}) by lack of memory of exponential distribution on $\mathbb{P}^\gamma$ we have $B^A(\infty)=e_1-Y$, where $Y$ has uniform
distribution $\xi(dx)=\frac{dx}{a}$ ($x\in [0,a]$) and $e_1$ is exponential r.v. with intensity $\lambda$.
This gives:
$$D=\frac{1}{a}\left(e^{a(\mu-\lambda)}-1\right)\frac{\mu-\lambda}{\mu}\frac{\lambda}{\mu}.$$
\end{ex}

\ack
This work is partially supported by the Ministry of Science and
Higher Education of Poland under the grants N N201 394137
(2009-2011).


\begin{thebibliography}{99}
\bibitem{2} {\sc Asmussen, S.} (2003). {\em Applied Probability and Queues (2nd ed.)}. Springer.
\bibitem{AP} {\sc Asmussen, S. and Pihlsg{\aa}rd, M.} (2007). Loss Rates for L\'{e}vy Processes with Two Reflecting Barriers.
{\em Math. Oper. Res.} {\bf 32(2)}, 308--321.
\bibitem{4} {\sc Bekker, R. and Zwart, B.} (2003). On an equivalence between loss rates and cycle maxima in
queues and dams. {\em Probab. Engrg. Inform. Sci.} {\bf 19}, 241--255.
\bibitem{9} {\sc Cooper, W.L., Schmidt, V. and Serfozo, R.F.} (2001). Skorohod-Loynes characterizations of queueing, fluid, and inventory processes. {\em Queueing Systems} {\bf 37}, 233-–257.
\bibitem{Dupuis1} {\sc Dupuis, P. and Ramanan, K.} (1999). Convex duality and the Skorokhod problem. I, II. {\em Probab. Theory Related Fields}  {\bf 115}, 153--195, 197--236.
\bibitem{10} {\sc Jelenkovi\'{c}, P.R.} (1999). Subexponential loss rates in a $GI/GI/1$ queue with applications.
{\em Queueing Systems} {\bf 33}, 91-–123.
\bibitem{11} {\sc Kella O. and Stadje, W.} (2004). A Brownian motion with two reflecting barriers and Markov-modulated speed. {\em J. Appl. Probab.} {\bf 41}, 1237-–1242.
\bibitem{12} {\sc Kella, O. and Whitt, W.} (1992). Useful martingales for stochastic storage processes with L\'{e}vy
input. {\em J. Appl. Probab.} {\bf 29}, 396-–403.
\bibitem{KOM}
{\sc Kella, O., Boxma, O. and Mandjes, M.} (2006). A L\'{e}vy process reflected at a Poisson age process.
{\em J. Appl. Probab.} {\bf 43(1)}, 221--230.
\bibitem{13} {\sc Kim, H.S. and Shroff, N.B.} (2001). On the asymptotic relationship between the overflow probability and the loss ratio. {\em Adv. Appl. Probab.} {\bf 33}, 836-–863.
\bibitem{17} {\sc Pihlsg{\aa}rd, M.} (2004). Loss rate asymptotics in a $GI/G/1$ queue with finite buffer.
{\em Stochastic Models} {\bf 21}, 913--931.
\bibitem{18} {\sc Pistorius, M.} (2003). On doubly reflected completely asymmetric L\'{e}vy processes. {\em Stoch.
Proc. Appl.} {\bf 107}, 131-–143.
\bibitem{miyazawa} {\sc Sakama, Y. and Miyazawa, M.} (2009). Asymptotic behaviors of the loss rate for Markov modulated fluid queue with a finite buffer. Submitted for publication.
\bibitem{23} {\sc Siegmund, D.} (1976). The equivalence of absorbing and reflecting barrier problems for stochastically monotone Markov processes. {\em Ann. Probab.} {\bf 4}, 914-–924.
\bibitem{Tanaka} {\sc Tanaka} (1979). H.  Stochastic differential equations with reflecting boundary conditions in convex
regions. {\em Hiroshima Math. J.} {\bf 9}, 163--177.
\bibitem{25} {\sc Zwart, B.} (2002). A fluid queue with a finite buffer and subexponential inputs. {\em Adv. Appl.
Probab.} {\bf 32}, 221-–243.


\end{thebibliography}
\end{document}